\documentclass[11pt]{amsart}
\usepackage[margin=3cm]{geometry}
\usepackage{comment}
\usepackage{enumitem}
\usepackage{float} 

\setlist[enumerate]{leftmargin=*, label=\alph*)} 
\setlist[itemize]{leftmargin=*} 

\usepackage{amssymb,amsthm,amsfonts,amsmath}

\newtheoremstyle{example-style}{5pt}{0pt}{}{}{\scshape}{:}{.5em}{}

\newtheorem{Thm}{Theorem}

\newtheorem{lem}[Thm]{Lemma}
\newtheorem{cor}[Thm]{Corollary}
\newtheorem{prop}[Thm]{Proposition}

\theoremstyle{example-style}

\theoremstyle{definition}

\newtheorem{Remark}[Thm]{Remark}

\begin{document}
\date{June 24, 2016}

\title{Cyclotomic polynomials at roots of unity}
\author{Bart{\l}omiej Bzd\c{e}ga, Andr\'es Herrera-Poyatos and Pieter Moree}
\date{August 2016}

\maketitle
{\def\thefootnote{}
}

\begin{abstract} 
The $n^{th}$ cyclotomic polynomial 
$\Phi_n(x)$ is the minimal polynomial of an $n^{th}$ primitive
root of unity. 
Hence $\Phi_n(x)$ is trivially zero at primitive $n^{th}$ roots of unity.
Using finite Fourier analysis we derive 
a formula for $\Phi_n(x)$ at the other roots of unity. This allows
one to explicitly evaluate
$\Phi_n(e^{2\pi i/m})$ with $m\in \{3,4,5,6,8,10,12\}$.
We use this evaluation with $m=5$ to give a simple reproof of a result 
of Vaughan (1975) on 
the maximum coefficient (in absolute value) of $\Phi_n(x)$. We also obtain a formula for $\Phi_n'(e^{2\pi i/m}) / \Phi_n(e^{2\pi i/m})$ with $n \ne m$, which is effectively applied to $m \in \{3,4,6\}$. Furthermore, we
compute the resultant of two
cyclotomic polynomials in a novel very short way.
\end{abstract}

\section{Introduction}
The study of cyclotomic polynomials $\Phi_n$ has a long and venerable history\footnote[1]{Even involving poems, e.g.
I. Schur's proof of the irreducibility of $\Phi_n(x)$ set 
to rhyme \cite[pp. 38-41]{Cremer}.}. In this paper we
mainly focus on two aspects: values at roots of unity and heights.
These two aspects are related. In order to explain the connection
we have to recall the notion of height.
Let $f(x)=a_0+a_1x+a_2x^2+\ldots+a_dx^d$ 
be a polynomial of degree $d=\deg f$. Then 
its height $H(f)$ is defined as $H(f)=\max_{0\le j\le d}|a_j|$. 
Now if $z$ is on the unit circle, then for $n>1$
we obviously have 
\begin{equation}
\label{An}
A_n:=H(\Phi_n)\ge \frac{\sum_{0\le j\le d}|a_n(j)|}{d+1}
\ge \frac{|\Phi_n(z)|}{\varphi(n)+1}
\ge \frac{|\Phi_n(z)|}{n},
\end{equation}
where 
$\Phi_n(x)=\sum_{j=0}^da_n(j)x^j$ 
and $d=\deg \Phi_n = \varphi(n)$, with $\varphi$ Euler's totient function.
This inequality shows that if we can pinpoint any $z$ on 
the unit circle for which $|\Phi_n(z)|$
is large, then we can obtain a non-trivial lower bound for
$A_n$ (cf. Bzd\c{e}ga \cite{Bzdega}).\\
\indent 
In this paper we show that there is an infinite sequence of integers $n$ such
that $|\Phi_n(z_n)|$ is large, with $z_n$ an appropriately chosen
primitive fifth root of unity. It is easy to deduce 
(see the proof of Theorem \ref{BB-T-Vaughan}) that for this sequence
$\log\log A_n \ge (\log2+o(1)){\log n}/{\log\log n}$ as $n$ tends to
infinity, which reproves a result of Vaughan \cite{Vaughan}. The infinite sequence
is found using Theorem \ref{BB-T-values}, our main result.\\
\indent We evaluate $\Phi_n(e^{2\pi i/m})$ for $m\in \{1,2,3,4,5,6\}$ and
every $n\ge 1$ in, respectively, Lemmas \ref{valueat1A}, \ref{valueat-1}, \ref{phini}, 
\ref{phinz3}, \ref{phinz5} and \ref{phinz6}. For $m\in \{1,2\}$ these results are
folklore and we recapitulate them for the convenience of the reader. 
For $m\in \{3,4,6\}$ the results were obtained by Motose
\cite{Motose}, but they need some 
small corrections (for details see
the beginning of Section \ref{sec:low-order}). We reprove these results using a different method which has the
advantage of reducing the number of cases being considered. Using a computer algebra package
we verified our results for $n\le 5000$. We note that
the field $\mathbb Q(e^{2\pi i/m})$ is of degree at most $2$ 
if and only if $m\in \{1,2,3,4,6\}$.\\
\indent Our main result expresses 
$\Phi_n(\xi_m)$, with $\xi_m$ an 
arbitrary primitive $m^{th}$ root of unity, in terms of the set of Dirichlet characters modulo $m$. 
This result allows one to explicitly evaluate $\Phi_n(\xi_m)$ also for $m\in \{5,8,10,12\}$  (values of
$m$ not covered
in the literature so far). 
\begin{Thm} 
\label{BB-T-values}
Let $n,m > 1$ be coprime integers. 
By $G(m)$ we denote the multiplicative group modulo $m$ and by 
$\widehat{G}(m)=\text{Hom} 
((\mathbb Z/m\mathbb Z)^*,\mathbb C^*)$ the set of Dirichlet characters modulo $m$.
For all $\chi\in \widehat{G}(m)$ let
$$C_\chi(\xi_m) = \sum_{g\in G(m)}\overline{\chi}(g)\log(1-\xi_m^g),$$
where we take the logarithm with imaginary part in $(-\pi,\pi]$. Then
$$\Phi_n(\xi_m) = \exp\left(\frac1{\varphi(m)}\sum_{\chi\in\widehat{G}(m)}
C_\chi(\xi_m)\chi(n)\prod_{p\mid n}(1-\overline{\chi}(p))\right).$$
\end{Thm}
The theorem is especially easy to use if $\widehat{G}(m)$ only consists of the trivial
and quadratic characters. This occurs precisely if $(\mathbb Z/m\mathbb Z)^*$ is a 
direct product of cyclic groups of order two. It is elementary to classify those 
$m$ and one finds that $m\in \{1,2,3,4,6,8,12,24\}$.\\
\indent A variant of Theorem \ref{BB-T-values} for
$\Phi_n'(\xi_m)/\Phi_n(\xi_m)$ is also obtained (Theorem \ref{BB-T-derivative}). It is used to evaluate $\Phi_n'(\xi_m)/\Phi_n(\xi_m)$ for $m \in \{3,4,6\}$. \\
\indent Kronecker polynomials are monic products of cyclotomic polynomials and a monomial. For them some of our results can be applied (see Section \ref{kronappl}).\\
\indent A question related to computing $\Phi_n(\xi_m)$ is that of determining its degree as 
an algebraic integer. This was considered in extenso by Kurshan and
Odlyzko \cite{Kurshan}. Their work uses Gauss and Ramanujan sums, the non-vanishing of Dirichlet L-series
at $1$, and the construction of Dirichlet 
characters with special properties.

\section{Preliminaries} \label{preli}
We recall some
relevant material on cyclotomic fields as several of our 
results can be reformulated in terms of 
cyclotomic fields. Most books on algebraic number theory contain a chapter on cyclotomic fields, 
for the advanced theory see, e.g., Lang \cite{Lang2}. Furthermore we consider
elementary properties of self-reciprocal polynomials and the
(generalized) Jordan totient function.\\
\indent The results in Section \ref{selfie} and Lemma \ref{lem:flauwmultiplicatief} in 
Section \ref{gejo} are our own, but
given their elementary nature they have been quite likely observed before. The proof
of Theorem \ref{thm:apostol} is new.

\subsection{Important notation}
We write double exponents not
as $a^{{b^c}}$, but as $(a)^{\wedge} b^c$ in
those cases where we think it enhances the readability.

Throughout we use the letters $p$ and $q$ to denote primes.
For a natural number $n$ we will refer to the exponent of $p$ in the
prime factorization of $n$ by $\nu_p(n)$, i.e., 
$p^{\nu_p(n)} \parallel n$. \\
\indent A primitive $n^{th}$ root of unity is a complex number $z$
satisfying $z^n=1$, but not $z^d=1$ for any $d<n$.
We let $\xi_n$ denote any primitive $n^{th}$ root
of unity. It is of the form $\zeta_n^j$ with $1\le j\le n$,
$(j,n)=1$ and $\zeta_n=e^{2\pi i/n}$.

\subsection{Cyclotomic polynomials}
In this section we recall some material on cyclotomic polynomials
we will need later in the paper. For proofs see, e.g., 
Thangadurai \cite{Thanga}.\\
\indent A definition of the $n^{th}$ cyclotomic polynomial is
\begin{equation}
\label{definitie}
\Phi_n(x)=\prod_{1\le j\le n,~(j,n)=1}(x-\zeta_n^j)\in 
\mathbb C[x].
\end{equation}
It is monic of degree $\varphi(n)$, has integer coefficients and is irreducible over $\mathbb Q$.
In $\mathbb Q[x]$ we have the factorization into irreducibles
\begin{equation}
\label{facintoirr}
x^n-1=\prod_{d \mid n}\Phi_d(x).
\end{equation}
By M\"obius inversion we obtain from this that
\begin{equation}
\label{phimoebius}
\Phi_n(x)=\prod_{d \, \mid n}(x^d-1)^{\mu(n/d)},
\end{equation}
with $\mu$ the M\"obius function.\\
\indent Lemma \ref{basiceqs}
and Corollary \ref{corbasiceqs} summarize some further properties of 
$\Phi_n(x)$. 
\begin{lem} \label{basiceqs}
We have\\
{\rm a)} $\Phi_{pn}(x)=\Phi_n(x^p)$ if $p\mid n$;\\
{\rm b)}  $\Phi_{pn}(x)=\Phi_n(x^p)/\Phi_n(x)$ if $p\nmid n$;\\
{\rm c)}  $\Phi_n(x)=x^{\varphi(n)}\Phi_n(1/x)$ for $n > 1$.
\end{lem}
\begin{cor} 
\label{corbasiceqs}
We have\\
$$\Phi_n(-x)=\begin{cases}(-1)^{\varphi(n)}\Phi_{2n}(x) & \hbox{if } 2\nmid n;\\
(-1)^{\varphi(n)}\Phi_{n/2}(x) & \hbox{if }2 \parallel n;\\
\Phi_n(x) & \hbox{if }4 \mid n.\end{cases}$$
\end{cor}

\subsection{Calculation of $\Phi_n(\pm 1)$}
\label{flauweevaluatie} 
The evaluation of $\Phi_n(1)$ is a classical result.
For completeness we formulate the result and give two proofs
of it, the first taken from Lang \cite[p. 74]{Lang}.

\begin{lem} 
\label{valueat1A}
We have
$$\Phi_n(1)=\begin{cases}0 & \hbox{if } n=1;\\
p & \hbox{if } n=p^e;\\
1 & \hbox{otherwise,}\end{cases}$$
with $p$ a prime number and $e\ge 1$.
\end{lem}
\begin{proof}
By \eqref{facintoirr} we have
\begin{equation}
\label{howtrivial!}
\frac{x^n-1}{x-1}=\prod_{d \mid n,~d>1}\Phi_d(x).
\end{equation}
Thus
\begin{equation}
\label{nisprod}
n=\prod_{d \mid n,~d>1}\Phi_d(1).
\end{equation}
We see that
$p=\Phi_p(1)$. 
Furthermore,
$p^f=\Phi_{p}(1)\Phi_{p^2}(1)\cdots \Phi_{p^f}(1)$. 
Hence, by induction
$\Phi_{p^f}(1)=p$. 
We infer that $\prod_{d\in {\mathcal Q},~d \mid n}\Phi_d(1)=n$,
where $\mathcal Q$ is the set of all prime
powers $>1$. Thus for the composite divisors $d$ of $n$, 
we have
$\Phi_d(1)=\pm 1$. 
Assume inductively that for $d \mid n$ and $d<n$ 
we have $\Phi_d(1)=1$. Then we see from our product that
$\Phi_n(1)=1$ too.
\end{proof}

The reader might recognize the von Mangoldt function 
$\Lambda$ in 
Lemma \ref{valueat1A}. Recall that the von Mangoldt function 
$\Lambda$ is defined as
$$\Lambda(n)=\begin{cases}\log p & \hbox{if } n=p^e,~e\ge 1;\\
0 & \hbox{otherwise.}\end{cases}$$
In terms of the von Mangoldt function we can reformulate
Lemma \ref{valueat1A} in the following way.
\begin{lem}
\label{valueat1B}
We have $\Phi_1(1)=0$. For $n>1$ we have
$\Phi_n(1)=e^{\Lambda(n)}$.
\end{lem}
We will give a reproof of this lemma in which the von Mangoldt
function arises naturally.
\begin{proof}[Proof of Lemma \ref{valueat1B}]
By M\"obius inversion the identity \eqref{nisprod} for all 
$n>1$ determines $\Phi_m(1)$ uniquely for all $m>1$. 
This means that it is enough to verify that
$\log n=\sum_{d \mid n,~d>1}\Lambda(d)$ for all $n>1$. Since
$\Lambda(1)=0$ it is enough to verify that
$\log n=\sum_{d \mid n}\Lambda(d)$ for all $n>1$. This is
a well known identity in elementary prime number theory.
\end{proof}
The Prime Number Theorem in the 
equivalent form $\sum_{n\le x}\Lambda(n)\sim x$
yields in combination with Lemma \ref{valueat1B} 
the following proposition.
\begin{prop}
The Prime Number Theorem is equivalent with the statement
that $$\sum_{2<n\le x}\log(\Phi_n(1))\sim x,~~x\rightarrow \infty.$$
\end{prop}
In a similar vein, Amoroso \cite{Amor}
considered a variant $h$ of the Mahler measure and established
that the estimate $h(\prod_{n\le x}\Phi_n)\ll x^{1/2+\epsilon}$ 
for every $\epsilon>0$ is equivalent
with the Riemann Hypothesis.

\subsection{Calculation of $\Phi_n(-1)$}
Once one has calculated 
$\Phi_n(1)$, the evaluation of $\Phi_n(-1)$ follows
on invoking Corollary \ref{corbasiceqs}.
\begin{lem} 
\label{valueat-1}
We have
\begin{equation*}
\Phi_n(-1)=\begin{cases}
-2 & \hbox{if } n=1; \\
0 & \hbox{if } n=2; \\
p & \hbox{if } n=2p^e ;\\
1 & \hbox{otherwise}.\end{cases}
\end{equation*}
with $p$ a prime number and $e\ge 1$.
\end{lem}
\begin{Remark}
It is also possible to prove this lemma along the
lines of the proof of Lemma \ref{valueat1A}, see Motose 
\cite{Motose}.
\end{Remark}

\subsection{Cyclotomic fields}
\label{cyclfields}
Several of the results in this paper can be rephrased in terms
of cyclotomic fields. A field is said to be cyclotomic if it is of the form $\mathbb Q[x]/(\Phi_m(x))$ for some $m\ge 1$. It
is isomorphic to $\mathbb Q(\zeta_m)$ 
which is the one obtained by adjoining $\zeta_m$ to $\Bbb Q$. It
satisfies $\left[\mathbb{Q}(\zeta_m) : \mathbb{Q}\right] = \deg \Phi_m = \varphi(m)$ and has $\mathbb Z[\zeta_m]$ as its ring of integers.

A field automorphism $\sigma$ of $\mathbb{Q}(\zeta_m)$ is completely determined by the image of $\zeta_m$. This has to be 
root of unity of order $m$ and hence $\sigma(\zeta_m) = \zeta_m^j$ with $1 \le j \le m$ and $(j,m) = 1$. 
It follows that Gal$(\mathbb{Q}(\zeta_m) / \mathbb{Q}) \cong \left(\mathbb{Z} / m\mathbb{Z}\right)^*$ and that
the norm of an algebraic number $\alpha$ in $\mathbb Q(\zeta_m)$ satisfies
\begin{equation} \label{galnorm}
N_{\mathbb Q(\zeta_m)/\mathbb Q}(\alpha) = \prod_{1 \le j \le m, (j,m) = 1} \sigma_j(\alpha),
\end{equation}
where $\sigma_j$ denotes the 
automorphism that sends $\zeta_m$ to $\zeta_m^j$.
It also follows that $\Phi_m(x)$, the minimal polynomial of $\zeta_m$, satisfies \eqref{definitie}.\\
\indent Let  $(j, m) = 1$.  We have
$\Phi_n(\zeta_m^j) = \Phi_n(\sigma_j(\zeta_m)) = \sigma_j(\Phi_n(\zeta_m))$ and
so in order to compute $\Phi_n(\zeta_m^j)$ it is enough to compute $\Phi_n(\zeta_m)$. In particular if and only if one of the values $\Phi_n(\zeta_m^j)$ is rational, then all of them are equal.

Let $k$ be an integer. On combining \eqref{galnorm} and \eqref{definitie} we infer that
\begin{equation}
\label{normo}
N_{\mathbb{Q}(\zeta_m) / \mathbb{Q}}(k-\zeta_m)=\Phi_m(k).
\end{equation}
\indent The resultant of two monic polynomials $f$ and $g$
having roots $\alpha_1,\ldots,\alpha_k$, respectively
$\beta_1,\ldots,\beta_l$ is given by
$$\rho(f,g)=\prod_{i=1}^k \prod_{j=1}^l (\alpha_i-\beta_j)=\prod_{i=1}^kg(\alpha_i).$$
In particular it follows from \eqref{definitie} that
\begin{equation}
\label{result}
\rho(\Phi_m,\Phi_n)=\prod_{1\le j\le m,~(j,m)=1}\Phi_n(\zeta_m^j).
\end{equation}
E. Lehmer (1930) \cite{Lehmer}, Diederichsen (1940) \cite{Died}, Apostol (1970)  and Louboutin (1997) \cite{Louboutin} all
computed the resultant of 
cyclotomic polynomials (see also  Sivaramakrishnan \cite[Chapter X]{Siva}). 
More recently Dresden
(2012) \cite{Dresden} gave yet another proof.
Here we present a very short new proof.
\begin{Thm}
\label{thm:apostol}
If $n>m>1$, then
$$\rho(\Phi_n,\Phi_m)=\begin{cases}
p^{\varphi(m)} & \text{if~}n/m=p^k \text{~for some prime }p
\text{~and~} k\ge 1;\\
1 & \text{otherwise}.
\end{cases}
$$
\end{Thm}
\begin{proof}

Assume that $n>m>1$. Then there exist a prime $p$
such that $\nu_p(n)>\nu_p(m)$. Put $n=Np^e$ and $m=Mp^f$
with $p\nmid M,N$. Obviously $e>f\ge0$.

Note that by the Chinese Remainder Theorem
every primitive residue $j$ modulo $m$
can be uniquely written as
$$j\equiv ap^f + bM \pmod{m},$$
where $a$ and $b$ are primitive residues
respectively modulo $M$ and $p^f$.
We will use this fact. In order to make the notation
shorter we will write $\prod_j$, $\prod_a$ and $\prod_b$
for the product over primitive residues
respectively modulo $m$, $M$ and $p^f$.

First we consider the case $M\neq N$. We have
\begin{align*}
\rho(\Phi_n,\Phi_m)
& = \prod_j\Phi_n(\zeta_m^j)
= \prod_j\frac{\Phi_N(\zeta_{Mp^f}^{jp^e})}
{\Phi_N(\zeta_{Mp^f}^{jp^{e-1}})}
= \prod_j\frac{\Phi_N(\zeta_M^{jp^{e-f}})}
{\Phi_N(\zeta_M^{jp^{e-f-1}})} \\
& = \prod_a\prod_b\frac{\Phi_N(\zeta_M^{ap^e+bMp^{e-f}})}
{\Phi_N(\zeta_M^{ap^{e-1}+bMp^{e-f-1}})}
= \left(\prod_a\frac{\Phi_N(\zeta_M^{ap^e})}
{\Phi_N(\zeta_M^{ap^{e-1}})}\right)^{\varphi(p^f)}
= 1.
\end{align*}

For $M=N$ we need to replace the quotients
by their limits. Using the L'H\^{o}pital rule
and the substitution $j\equiv ap^f + bM$, we obtain
$$\Phi_n(\zeta_m^j) = \lim_{z\to\zeta_m^j}
\frac{\Phi_N(z^{p^e})}{\Phi_N(z^{p^{e-1}})}
= p\zeta_M^{\varphi(p^e)}
\frac{\Phi_N'(\zeta_M^{ap^e})}
{\Phi_N'(\zeta_M^{ap^{e-1}})}.$$

After taking the product over all primitive
$a$ modulo $m$, the derivatives cancel out.
So for $M=N$ we have
$$\rho(\Phi_n,\Phi_m)
= \prod_a\prod_b(p\zeta_M^{a\varphi(p^e)})
= p^{\varphi(m)}\zeta_M^{\varphi(p^e)\varphi(p^f)\sum_aa}
= p^{\varphi(m)},$$
where we used that $M\mid\sum_aa$ for $M>2$. If $M=2$ then
$\zeta_M=-1$, $p>2$ and $\varphi(p^e)$ is even.
\end{proof}
\begin{cor}
\label{phiunit}
Let $n>m>1$. The algebraic integer $\Phi_n(\zeta_m)$ is not a unit in $\mathbb Z[\zeta_m]$ if and only
if $n/m$ is a prime power.
\end{cor}

\subsection{Self-reciprocal polynomials}
\label{selfie}
A polynomial $f$ of degree $d$ is said to be self-reciprocal if
$f(x)=x^df(1/x)$. If $f(x)=-x^df(1/x)$, then $f$ is said to be anti-self-reciprocal. Lemma \ref{basiceqs}c says 
that $\Phi_n$ is self-reciprocal for $n\ge2$. Note that $\Phi_1$ is
anti-self-reciprocal.

\begin{lem} \label{selfreciprocal-evaluation}
Let $f \in \mathbb{R}[x]$ be a self-reciprocal polynomial. Then for $|z|=1$ we have
$$f(z) = \pm |f(z)| z^{\frac{\deg f}{2}}.$$
If $f \in \mathbb{R}[x]$ is an anti-self-reciprocal polynomial, then for $|z|=1$ we have
$$f(z) = \pm i|f(z)| z^{\frac{\deg f}{2}}.$$
\end{lem}
\begin{proof}
Let $d=\deg f$. If $f$ is self-reciprocal and $|z|=1$ we have $f(z)=z^df(1/z)=z^d\overline{f(z)}$. Multiplying both sides by $f(z)$ and taking the square root we obtain the first claim. \\
\indent If $f$ is anti-self-reciprocal and $|z|=1$  we 
have  $f(z) = -z^df(1/z)=-z^d\overline{f(z)}$ and the proof is analogous.
\end{proof}

The behaviour of a self-reciprocal $f$ and its first 
derivative at $\pm 1$ is easily determined.
\begin{prop} \label{self-reciprocal:deriv}
Let $f$ be a polynomial of degree $d\ge 1$.\\
Suppose that $f$ is self-reciprocal.\\
{\rm a)} We have $f'(1) = f(1) d / 2$;\\
{\rm b)} If $2\nmid d$, then $f(-1)=0$. If $2\mid d$, then $f'(-1)=-f(-1)d/2$.\\
Suppose that $f$ is anti-self-reciprocal.\\
{\rm a)} We have $f(1) =0$;\\
{\rm b)} If $2\mid d$, then $f(-1)=0$. If $2\nmid d$, then $f'(-1)=-f(-1)d/2$.
\end{prop}
\begin{proof}
If $f$ is self-reciprocal, then $f(z)=z^df(1/z)$. If $f$ is
anti-self-reciprocal we have $f(z)=-z^df(1/z)$. Differentiating both sides and
substituting $z=\pm 1$ gives 
the result.
\end{proof}
The next result concerns the behaviour 
of self-reciprocal polynomials in roots of unity 
other than $\pm 1$.
\begin{lem} \label{selfreciprocal-int}
Let $f \in \mathbb{Z}[x]$ be a self-reciprocal polynomial of even degree $d$ and $m \in \{3,4,6\}$. Then $\xi_m^{-d/2}f(\xi_m)$ is an integer.
\end{lem}
\begin{proof}
For any $m$ with $\varphi(m)=2$ the field
$\mathbb Q(\xi_m)$ is quadratic. Hence we
can write $\xi_m^{-d/2}f(\xi_m)=a+b\xi_m$ with $a$
and $b$ integers. Since by assumption $f$ is self-reciprocal we have
$a+b\xi_m^{-1}=\xi_m^{d/2}f(\xi_m^{-1})
=\xi_m^{-d/2}f(\xi_m)=a+b\xi_m$. Hence $b=0$ and
the result follows.
\end{proof}

\subsection{The (generalized) Jordan totient function}
\label{gejo}
Let $k\ge 1$ be an integer. The $k^{th}$ Jordan totient
function is defined by
$$J_k(n)=\sum_{d\, \mid n}\mu(\frac{n}{d})d^k.$$
As $J_k$ is a Dirichlet convolution of multiplicative functions, it is itself multiplicative. One has
$$J_k(n)=n^k\prod_{p\, \mid n}(1-\frac{1}{p^k}).$$
\indent Given a character $\chi$ and an integer
$k\ge 0$ we define
\begin{equation}
\label{eq:phianalogue}
J_k(\chi;n)=\sum_{d\, \mid n}\mu(\frac{n}{d})d^k\chi(d).
\end{equation}
Since $J_k(\chi;\cdot)$ is a Dirichlet convolution
of multiplicative functions, it is 
a multiplicative function itself. The next lemma demonstrates that 
it is an analogue of the Jordan totient function. Recall
that $\mathrm{rad}(n)=\prod_{p\, \mid n}p$ is the radical, sometimes
also called the squarefree kernel, of $n$.
\begin{lem}
\label{lem:flauwmultiplicatief}
Let $\chi$ be a character modulo $m$ and $k\ge 0$ an 
integer. We have
$$J_k(\chi;n)=\prod_{p^e \parallel n}p^{k(e-1)}\chi(p^{e-1})(p^k\chi(p)-1)=\left(\frac{n}{\mathrm{rad}(n)}\right)^k
\chi\left(\frac{n}{\mathrm{rad}(n)}\right)\prod_{p\, \mid n}(p^k\chi(p)-1).$$
If $n$ is squarefree, then 
$J_k(\chi;n)=\prod_{p\, \mid n}(p^k\chi(p)-1)$. 
If  $(m,n)=1$, then
$$J_k(\chi;n)=\chi(n)n^k\prod_{p \, \mid n}\left(1-\frac{\overline{\chi}(p)}{p^k}\right).$$
\end{lem}
\begin{proof}
The proof follows by the usual arguments from the elementary theory
of arithmetic functions.
\end{proof}

\section{Cyclotomic values in arbitrary roots of unity} \label{sec:arbitrary-roots}

Let us consider two positive integers $n,m$ with $n > 1$ and $m \ge 1$. In this section we present general facts about the value $\Phi_n(\xi_m)$. Clearly $\Phi_n(\xi_m) = 0$ if and only if $n = m$. Hence we study the case $n \ne m$.\\
\indent The next result is due to Kurshan and Odlyzko \cite[Corollary 2.3]{Kurshan}.
We give a simpler reproof of it (suggested to
us by Peter Stevenhagen).
\begin{lem} \label{lem:is-real}
Let $n\ge 2$. The cyclotomic value
$\Phi_n(\xi_m)$ is non-zero and real if and only if $m \mid \varphi(n)$.
\end{lem}
\begin{proof}
The number $\Phi_n(\xi_m)$ is real if and only if $\Phi_n(\xi_m)=\overline{\Phi_n(\xi_m)}=\Phi_n(\xi_m^{-1})$.
By the self-reciprocity of $\Phi_n$ we see that this is equivalent with
$\Phi_n(\xi_m)=\xi_m^{-\varphi(n)}\Phi_n(\xi_m)$, which is equivalent
with $n=m$ or $m|\varphi(n)$. On noting that $n\nmid \varphi(n)$ and 
$\Phi_n(\xi_m)=0$ if and only if $n=m$, the proof is completed.
\end{proof}

Lemma \ref{selfreciprocal-evaluation} shows that for $n \ge 2$ we have
$\Phi_n(\xi_m) =\pm  |\Phi_n(\xi_m)|\xi_m^{\varphi(n) / 2}$.
The next result shows that the sign is given by $(-1)^{\varphi(n/m;n)}$, where
$\varphi(x;n)$ is the number of positive integers $j\le x$ with
$(j,n)=1$.
\begin{lem}\label{lem:phin-eval}  
Write $\xi_m = \zeta_m^j$. For $n \ge 2$ we have $\Phi_n(\xi_m) = (-1)^{\varphi(nj/m;n)}|\Phi_n(\xi_m)|\xi_m^{\varphi(n) / 2}$.
\end{lem}
\begin{proof}
Let us consider the function $g(t) = e^{-it\varphi(n) / 2} \Phi_n(e^{it})$ with $t \in [0, 2\pi)$. 
The self reciprocity of $\Phi_n$ ensures that $g(t)$ is
invariant under conjugation and hence real.
Note that $g$ is differentiable. Furthermore, the set of roots of $g$ 
equals $\{2\pi j/n: 1 \le j < n, (j,n) = 1\}$. All of the roots are 
simple. Since $g(0)=\Phi_n(1) > 0$ we infer that $g(t) = (-1)^{\varphi(n t / (2\pi);n)}|\Phi_n(e^{it})|$, which by substituting $t=2\pi j/m$ yields the result.
\end{proof}

\begin{cor}\label{cor:one-case}
Write $\xi_m = \zeta_m^j$.
In case $\Phi_n(\xi_m) \in \{-1,1\}$ for some $n\ge 2$, then we have
$\Phi_n(\xi_m) = (-1)^{\varphi(nj/m; n)+j\varphi(n) / m }$.
\end{cor}
\begin{proof}
Write $\xi_m = \zeta_m^j$. If $\Phi_n(\xi_m) \in \{-1,1\}$, then 
$\Phi_n(\xi_m) = (-1)^{\varphi(nj/m;n)}\xi_m^{\varphi(n) / 2}$.
By Lemma  \ref{lem:is-real}  
we obtain that $m \mid \varphi(n)$. 
The result now follows on noting that $\xi_m^{\varphi(n) / 2} = (-1)^{j\varphi(n) / m}$. \qedhere
\end{proof}

\begin{lem} \label{mod1}
Let us assume that there exists $p \equiv 1 \pmod m$ and $k \ge 1$ such that $n = p^k n'$ with $p \nmid n'$.\\
{\rm a)} If $n' \ne m$, then $\Phi_n(\xi_m) = 1$.\\
{\rm b)} If $n'= m$, then $\Phi_n(\xi_m) = p$.
\end{lem}
\begin{proof}$~~$\\
a) We have $\Phi_n(x) = \Phi_{n'}(x^{p^k}) / \Phi_{n'}(x^{p^{k-1}})$ due to Lemma \ref{basiceqs}. 
By noting that $\xi_m^p = \xi_m$ it follows that
\[\Phi_n(\xi_m) = \frac{\Phi_{n'}(\xi_m^{p^k})}{\Phi_{n'}(\xi_m^{p^{k-1}})} = 1.\]
b) We apply L'H\^opital's rule and obtain 
\[\Phi_n(\xi_m) = \frac{p^k \xi_m^{p^k - 1}\Phi_m'(\xi_m^{p^k})}{p^{k-1} \xi_m^{p^{k-1} - 1}\Phi_m'(\xi_m^{p^{k-1}})} = p.\qedhere\]
\end{proof}

A version of Lemma \ref{mod1} has already been stated by Motose \cite[Section 4]{Motose}. Nonetheless, it contains a mistake since his lemma claims that $\Phi_n(\xi_m) = 1$ for case b).

\begin{lem} \label{mod-1}
Let us assume that there exists $p \equiv -1 \pmod m$ and $k \ge 1$ such that $n = p^k n'$ with $p \nmid n'$. \\
{\rm a)}  If $n' = 1$, then $\Phi_n(\xi_m) = -\xi_m^{(-1)^k}$.\\
{\rm b)}  If $n' \ne m$, then $\Phi_n(\xi_m) = \xi_m^{(-1)^k\varphi(n')}$. Furthermore, if $n' \ge 3$, then $\Phi_n(\xi_m) = \xi_m^{\varphi(n)/2}$.\\
{\rm c)}  If $n' = m$, then $\Phi_n(\xi_m) = - p \xi_m^{(-1)^k \varphi(m)}$.
\end{lem}
\begin{proof}$~~$\\
a) By \eqref{facintoirr} we have 
$$ \Phi_{p^k}(\xi_m) = \frac{\xi_m^{p^k} - 1}{\xi_m^{p^{k-1}} - 1}.$$
Assertion a) is easily established on noting that $\xi_m^{p^k} = \xi_m^{(-1)^k}$.\\
b) We have $\Phi_n(x) = \Phi_{n'}(x^{p^k}) / \Phi_{n'}(x^{p^{k-1}})$. In light of the self-reciprocity of $\Phi_{n'}$ we find that
\[ \Phi_n(\xi_m) = \frac{\Phi_{n'}\left(\xi_m^{(-1)^k}\right)}{\Phi_{n'}\left(\xi_m^{(-1)^{k+1}}\right)} = \xi_m^{(-1)^k\varphi(n')}. \]
Furthermore, if $n'\ge 3$, then 
$\varphi(n) / 2 = p^{k-1} (p-1) \varphi(n')/2 \equiv (-1)^k \varphi(n') \pmod m$.\\
c) L'H\^opital's rule yields
\[\Phi_n(\xi_m) = \frac{p^k \xi_m^{p^k - 1}\Phi_m'(\xi_m^{p^k})}{p^{k-1} \xi_m^{p^{k-1} - 1}\Phi_m'(\xi_m^{p^{k-1}})} = p \xi_m^{2(-1)^k} \frac{\Phi_m'(\xi_m^{(-1)^k})}{\Phi_m'(\xi_m^{(-1)^{k+1}})}.\]
Assertion c) follows on differentiating the equality $\Phi_m(z) = z^{\varphi(m)} \Phi_m(1/z)$ giving rise to
\[\Phi_m'(\xi_m^{(-1)^k}) = - \xi_m^{(-1)^k(\varphi(m)-2)} \Phi_m'(\xi_m^{(-1)^{k+1}}). \qedhere\]
\end{proof}

\section{The values - general method} \label{sec:values:general-method}

In this section we present a general method of computing $\Phi_n(\xi_m)$. 
Our first step is to reduce to the case 
where $m$ is coprime to $n$. In order to do this, we write $n=n_1n_2$, where $n_1=\prod_{p^e \parallel n,~p\,\nmid \,m}p^e$ is the largest divisor of $n$ which is coprime to $m$. By the equations of Lemma \ref{basiceqs} and by an induction on $n_2$ we have
\begin{equation} \label{eq:simp-coprime}
\Phi_n(\xi_m) = \prod_{d\,\mid n_2}\Phi_{n_1}(\xi_m^d)^{\mu(n_2/d)}.
\end{equation}
For small $n_2$ this formula is quite effective.

Therefore throughout this section we assume that $m,n>1$ are coprime.

\begin{proof}[Proof of Theorem \ref{BB-T-values}]
Note that $\log(1-\xi_m^d)$ considered as a function $d$ is periodic with period $m$ and so it can be treated as a 
function $G(m)\to\mathbb{C}$. It follows that
$$\log(1-\xi_m^d) = \frac{1}{\varphi(m)}\sum_{\chi\in\widehat{G}(m)}C_\chi(\xi_m)\chi(d).$$
We find that
$\log\Phi_n(\xi_m)$,  up to a multiple of $2\pi i$, equals
$$\sum_{d\,\mid n}\mu(\frac{n}{d})\log(1-\xi_m^d) =\frac{1}{\varphi(m)}\sum_{\chi\in\widehat{G}(m)}C_\chi(\xi_m)
J_0(\chi;n).$$
The proof is completed by invoking 
Lemma \ref{lem:flauwmultiplicatief} with $k=0$.
\end{proof}

\begin{Remark}
A character $\chi$ may be omitted if there exists a prime $p\mid n$ for which $\overline{\chi}(p)=1$. 
In particular, the principal character may be omitted.
It makes computing
$\Phi_n(\xi_m)$ using Theorem \ref{BB-T-values} a less daunting task.

If we wish only to compute $|\Phi_n(\xi_m)|$, then in addition we may omit all characters $\chi$ satisfying $\chi(-1)=-1$, since for such characters we have
$${C_\chi}(\xi_m)= \frac{1}{2}\sum_{g\in G(m)}
\left(\overline{\chi}(g)\log(1-\xi_m^g)+\overline{\chi}(-g)\log(1-\xi_m^{-g})\right)
\in i\mathbb R.$$
\end{Remark}

\begin{cor}
Let $(m,n)=1$ and $n>1$. \\
{\rm a)} If $n$ has any prime divisor $q$ congruent to $1$ modulo $m$, then $\Phi_n(\xi_m)=1$.\\
{\rm b)} If $m\in\{3,4,6\}$ and $n$ has no prime divisor congruent to $1$ modulo $m$ then 
\[\Phi_n(\xi_m)=-(\xi_m)^{~\wedge} (-(-2)^{\omega(n)-1}).\]
\end{cor}

\begin{proof}$~$\\ a) As $\overline{\chi}(p)=\overline{\chi}(1)=1$  we have $\prod_{p\mid n}(1-\overline{\chi}(p))=0$, so $\Phi_n(\xi_m)=e^0=1$. \\
b) Note that there is only one non-principal character $\chi$ and 
that it satisfies $\chi(-1)=-1$. Therefore $\prod_{p\mid n}(1-\overline{\chi}(p)) = 2^{\omega(n)}$ and $C_\chi(\xi_m) = \log(1-\xi_m) - \log(1-\xi_m^{-1})$. It follows that
$$\Phi_n(\xi_m) = \exp\left(\frac12(\log(1-\xi_m) - \log(1-\xi_m^{-1}))\chi(n)2^{\omega(n)}\right)
= -(\xi_m)^{~\wedge} (-(-2)^{\omega(n)-1}),$$
as desired.
\end{proof}

\begin{cor} \label{BB-C-values5}
Let $m \in \{5,8,10,12\}$ and $n > 1$ be coprime 
with $m$. Suppose that $n$ has no prime divisor $\pm 1 \pmod m$. Then
$$\log|\Phi_n(\xi_m)| = (-1)^{\Omega(n)-1}2^{\omega(n)-1}\log|\gamma_m|,$$
where 
$$\gamma_m=\begin{cases}1 + \xi_m & \hbox{if } m=5;\\
1+\xi_m+\xi_m^2 & \hbox{if } m \in \{8,10\};\\
1+\xi_m+\xi_m^2+\xi_m^3+\xi_m^4 & \hbox{if } m=12.\end{cases}$$
\end{cor}

\begin{proof}
The only non-principal character for which $C_\chi(\xi_m)$ has non-zero real part is the quadratic character $\chi$. We have $\Re C_\chi(\xi_m) = -2\log|\gamma_m|$ and by Theorem \ref{BB-T-values}
$$\log|\Phi_n(\xi_m)| = -\frac12(\log|\gamma_m|)~\chi(n)~\prod_{p\mid n}(1-\overline{\chi}(p)).$$
The assumption on $n$ we made implies that $\overline{\chi}(p)=-1$ for all $p\mid n$ 
and hence $\chi(n)=(-1)^{\Omega(n)}$ and so
$\prod_{p\mid n}(1-\overline{\chi}(p))=2^{\omega(n)}$.
\end{proof}

\section{Cyclotomic values in roots of unity of low order}\label{sec:low-order}
In this section we apply the obtained results
in order to easily compute $\Phi_n(\zeta_m)$ for $m \in \{3,4,5,6\}$. 
For $m\in \{3,4,6\}$ these values have already been computed by Motose \cite{Motose}. However, some of the 
results in Section \ref{sec:arbitrary-roots} allow us to provide shorter proofs. For $m\in \{1,2\}$ the computation is folklore and it was
discussed in Section \ref{flauweevaluatie}.\\
\indent In \cite{Motose} there are some
inaccuracies. As we mention in Section \ref{sec:arbitrary-roots} in part (1) 
of the first lemma one has
also to require that $m \ne l$. This oversight leads to
the incorrect assertion in Proposition 3 that if
$p\equiv 1 \pmod{3}$ for some prime divisor $p$ of
$m$, then $\Phi_n(\zeta_3)=1$. This is false as $\Phi_{3p^k}(\zeta_3)=p$.
A similar remark applies to Proposition 4, where
$\Phi_{6p^k}(\zeta_6)=p$, rather than 1 as claimed.
In the statement of Proposition 3 part (2) one has
to read $l+k$ instead of $l+k-1$. As the proof
is carried out correctly, this is a typo. As consequence
of the typo in Proposition 3, the exponent in 
case (6) in Proposition 4 is computed to be $l+s+k$ instead
of $l+s+k-1$.

\subsection{Calculation of $\Phi_n(i)$}
Lemma \ref{mod1} and Lemma \ref{mod-1} reduce the number of possible cases. Hence it is not difficult to establish the following result.
\begin{lem}
\label{phini}
We have $ \Phi_n(i)=1 $ except for the cases listed in the table below.
\begin{center}
\begin{tabular}{|c|c|}
\hline
$ n $ & $ \Phi_n(i) $  \\
\hline
$1$ & $i-1$  \\
\hline
$2$ & $i+1$  \\
\hline
$4$ & $0$  \\
\hline
$ 4p^k $ & $ p $\\
\hline
$ p_3^k $ & $ (-1)^{k+1}i $\\
\hline
$ 2p_3^k $ & $ (-1)^k i $\\
\hline
$ p_3^kq_3^l, ~2p_3^kq_3^l $ & $ -1 $\\
\hline
\end{tabular}
\end{center}
Here $ p,p_3$ and $q_3 $ are primes such that 
$p_3\ne q_3$ and $ p_3\equiv q_3\equiv 3\pmod 4 $. 
Furthermore, $ k$ and $l $ are arbitrary positive integers.
\end{lem}
\begin{proof}
The first three entries of the table follow by 
direct computation and hence we may assume that $n \ne 1,2,4$.
\begin{itemize}
\item In case $4 \mid n$ we 
have $\Phi_n(i) = \Phi_{n/2}(-1)$ and, by applying Lemma \ref{valueat-1}, $\Phi_n(i)$ is seen to equal $p$ if $n = 4p^k$ 
and $1$ otherwise.
\item In case $4 \nmid n$, we separately consider
three subcases:
\begin{enumerate}
\item The integer $n$ has a prime factor $p\equiv 1 \pmod 4$.\\
By Lemma \ref{mod1} we have $\Phi_n(i)=1$.
\item The integer $n$ is odd and has no prime factor  
$p\equiv 1 \pmod 4$.\\ 
Thus we can write $n=q_1^{e_1} \cdots q_r^{e_r}$ with $q_j \equiv 3 \pmod 4$ 
and $e_j \ge 1$ for every $1\le j\le r$. 
By  Lemma \ref{mod-1} it follows that $\Phi_{n}(i) = (-1)^{e_1+1}i$ if $r = 1$, $\Phi_n(i) = -1$ if $r=2$ and $\Phi_n(i) = 1$ otherwise.
\item The integer $n$ is even and has no prime factor  
$p\equiv 1 \pmod 4$.\\ 
Note that $\Phi_n(i) = \Phi_{n/2}(-i) = \overline{\Phi_{n/2}(i)}$ and hence the result follows from subcase b).
\end{enumerate}
\end{itemize}
Since we have covered all cases, the proof is concluded.
\end{proof}

\subsection{Calculation of $\Phi_n(\zeta_3)$}
\begin{lem}
\label{phinz3}
We have $ \Phi_n(\zeta_3)=1 $ except for the cases listed in the table below.
\begin{center}
\begin{tabular}{|c|c|}
\hline
$ n $ & $ \Phi_n(\zeta_3) $  \\
\hline
$1$ & $\zeta_3-1$  \\
\hline
$3$ & $0$  \\
\hline
$3 p^{k}$ & $p$  \\
\hline
$ q^{k} $ & $ -1/\zeta$\\
\hline
$ 3q^{k} $ & $ -q\zeta$\\
\hline
$ q_1^{e_1} \ldots q_r^{e_r},~r\ge 2$ & $ 1/{\zeta} $\\
\hline
$ 3q_1^{e_1} \ldots q_r^{e_r},~r\ge 2 $ & $\zeta$\\
\hline
\end{tabular}
\end{center}
Here $p\not\equiv 2 \pmod{3}$ is a prime.
The integers $q$ and $q_1, \ldots, q_r$ are primes congruent 
to $2$ modulo $3$ with $r\ge 2$ and $q_1, \ldots, q_r$ distinct. 
Furthermore, $k$ and $e_1, \ldots, e_r$ are arbitrary positive 
integers
and $\zeta=(\zeta_3)^{\wedge} (-1)^{s}$
with  $s=\Omega(n)-\omega(n)=\Omega(n/\mathrm{rad}(n))$.
\end{lem}
\begin{proof}
The first two entries of the table follow by 
direct computation and hence we may assume that $n \ne 1,3$.
\begin{itemize}
\item In case $9 \mid n$ we 
have $\Phi_n(\zeta_3) = \Phi_{n/3}(1)$ by invoking Lemma \ref{basiceqs}.
This yields $3$ if $n$ is a power of $3$ and 1 otherwise.
\item In case $9 \nmid n$ we separately consider three subcases:
\begin{enumerate}
\item The integer $n$ has a prime factor $p \equiv 1 \pmod 3$.\\
Lemma \ref{mod1} yields $\Phi_n(\zeta_3) = p$ if $n = 3p^k$ and $1$ otherwise.
\item The integer $n$ has no prime factor $p \equiv 1 \pmod 3$ and
$3 \nmid n$.\\
Thus we can write $n=q_1^{e_1} \cdots q_r^{e_r}$ with $q_j \equiv 2 \pmod 3$ 
and $e_j \ge 1$ for every $1\le j\le r$.\\
We distinguish two cases:\\
$r = 1$. By Lemma \ref{mod-1}a it follows that
$\Phi_{q_1^{e_1}}(\zeta_3) = -(\zeta_3)^{\wedge} (-1)^{s-1}=-1/\zeta$.\\ 
$r \ge 2$. On applying Lemma \ref{mod-1}b  
we obtain $\Phi_n(\zeta_3) = (\zeta_3)^{\wedge} (-1)^{s+1}=1/\zeta$.
\item The integer $n$ has no prime factor $p \equiv 1 \pmod 3$ and
$3 \mid n$.\\
Note that $\Phi_n(\zeta_3) = {\Phi_{n/3}(1)}/{\Phi_{n/3}(\zeta_3)}$ as a consequence of Lemma \ref{basiceqs}. Hence the result follows from the subcase b) and Lemma \ref{valueat1A}.\qedhere
\end{enumerate}
\end{itemize}
\end{proof}
\subsection{Calculation of $\Phi_n(\zeta_6)$}
In our computation of $\Phi_n(\zeta_6)$ we make freely use of the fact that
 $-\zeta_3 = \zeta_6^{-1}$.
\begin{lem}
\label{phinz6}
We have $ \Phi_n(\zeta_6)=1 $ except for the cases listed in the table below.

\begin{center}
\begin{tabular}{|c|c|}
\hline
$ n $ & $ \Phi_n(\zeta_6) $  \\
\hline
$1$ & $\zeta_3$  \\
\hline
$2$ & $\zeta_6+1$  \\
\hline
$3$ & $2\zeta_6$  \\
\hline
$6$ & $0$  \\
\hline
$6p^k$ & $p$  \\
\hline
$2q^k$ & $-\zeta$\\
\hline
$6q^k$ & $-q/\zeta$\\
\hline
$q_1^{e_1} \ldots q_r^{e_r}$ (different from $2$ and $2q^k$) & $\zeta$\\
\hline
$ 3q_1^{e_1} \ldots q_r^{e_r} $ (different from $6$ and $6q^k$) & $1/\zeta$\\
\hline
\end{tabular}
\end{center}
Here $p$ is $3$ or a prime number congruent to $1$ modulo $6$. The integers $q$ and $q_1, \ldots, q_r$ are $2$ or primes congruent to $5$ modulo $6$ with $r \ge 1$ and $q_1,\ldots,q_r$ distinct. 
Furthermore $k$ and $e_1, \ldots, e_r$ are arbitrary positive integers 
and $\zeta=(\zeta_3)^{\wedge} (-1)^{s}$
with  $s=\Omega(n)-\omega(n)=\Omega(n/\mathrm{rad}(n))$.
\end{lem}
\begin{proof}
The first four entries of the table follow by 
direct computation and hence we may assume that $n \ne 1,2,3,6$.
\begin{itemize}
\item In case $\nu_3(n) \ge 2$ we have $\Phi_n(\zeta_6) = \Phi_{n/3}(-1)$, which yields $3$ if $n = 6 \cdot 3^{k}$ and 1 otherwise.
\item In case $\nu_3(n) \le 1$ we separately consider three subcases:
\begin{enumerate}
\item The integer $n$ has a prime factor $p \equiv 1 \pmod 6$.\\
By Lemma \ref{mod1} we obtain $p$ if $n = 6 p^k$ and $1$ otherwise.
\item The integer $n$ has no prime factor $p \equiv \pm 1 \pmod 6$.\\ There are two possibilities:
\begin{enumerate}[label=\roman*)]
\item $n = 2^{k+1}$. We have $\Phi_n(\zeta_6) = \Phi_{n/2}(\zeta_3) = -(\zeta_3)^{\wedge} {(-1)^{k}} = -\zeta$.
\item $n = 6 \cdot 2^k$. We have $\Phi_n(\zeta_6) = \Phi_{n/2}(\zeta_3) = - 2 (\zeta_3)^{\wedge}{(-1)^{k+1}} = -2 / \zeta$.
\end{enumerate}
\item The integer $n$ has no prime factor $p \equiv 1 \pmod 6$ and it has a prime factor $q \equiv -1 \pmod 6$. \\
There are three possibilities:
\begin{enumerate}[label=\roman*)]
\item $n = q^k$. Lemma \ref{mod-1}a yields $\Phi_n(\zeta_6) = - (\zeta_6)^{\wedge}{(-1)^{k}} = \zeta$.
\item $n = q^k n'$ with $1 < n' \ne 6$ and $q \nmid n'$. Lemma \ref{mod-1}b implies $\Phi_n(\zeta_6) = (\zeta_6)^{\wedge}((-1)^k \varphi(n'))$. Thus we have $\Phi_{2q^k}(\zeta_6) = -\zeta$. Let us assume $n' > 2$. Now we compute $\zeta_6^{\varphi(n')}$.
\begin{itemize}
\item If $3 \nmid n'$, then $\zeta_6^{\varphi(n')} = \zeta_3^{\varphi(n')/2} = (\zeta_3)^{\wedge}{(-1)^{\Omega(n')-\omega(n')+1}}$ and $\Phi_n(\zeta_6) = \zeta$.
\item If $3 \mid n'$, then $\zeta_6^{\varphi(n')} = \zeta_3^{\varphi(n'/3)} = (\zeta_3)^{\wedge}{(-1)^{\Omega(n')-\omega(n')}}$ and $\Phi_n(\zeta_6) = 1/\zeta$.
\end{itemize}
\item $n = 6 q^k$. We have $\Phi_n(\zeta_6) = \Phi_{n/3}(-1) / \Phi_{n/3}(\zeta_6) = -q/\zeta$.
 \qedhere
\end{enumerate}
\end{enumerate}
\end{itemize}
\end{proof}

\begin{lem}
\label{lem:tablecombi}
Let $m\in \{1,2,3,4,6\}$ and $n>m$ be integers. Then
$$|\Phi_n(\xi_m)|=
\begin{cases}
p & \text{if~} n/m=p^k \text{~is a prime power};\\
1 & \text{otherwise}.
\end{cases}
$$
\end{lem}
\begin{proof}
For $m=1,2$ the result follows by Lemma
\ref{valueat1A}, respectively Lemma \ref{valueat-1}. So we may assume 
that $m\in \{3,4,6\}$ (and so $n>m\ge 3$). 
Since $\deg \Phi_n=\varphi(n)$ is even for $n\ge 4$, it follows
by Lemma \ref{selfreciprocal-int}  that $\Phi_n(\zeta_m)=\zeta_m^{\varphi(n)/2}a$, with $a$ an integer.
Letting the Galois automorphisms of $\mathbb Q(\zeta_m)/\mathbb Q$ act on both sides of this identity
we see that $|\Phi_n(\xi_m)|$ is an integer that is independent
of the specific choice of $\xi_m$. The result
now follows from Theorem \ref{thm:apostol} and 
the identity \eqref{result}. 
\end{proof}
\begin{Remark}
Lemma \ref{lem:tablecombi} can also 
be deduced from Lemmas \ref{valueat1A}, \ref{valueat-1}, \ref{phini}, \ref{phinz3} and \ref{phinz6}.
\end{Remark}

\subsection{Calculation of $\Phi_n(\zeta_5)$}
We use Corollary \ref{BB-C-values5} and Lemma \ref{lem:phin-eval} 
to compute $\Phi_n(\zeta_5)$. One can also compute $\Phi_n(\zeta_m)$ for $m \in \{8,10,12\}$ by a similar procedure. Nonetheless, the number of possible cases significantly increases for those values of $m$. 
\begin{lem}
\label{phinz5}
We have $\Phi_n(\zeta_5)=1$ except for the cases listed in the table below.

\begin{center}
\begin{tabular}{|c|c|}
\hline
$ n $ & $ \Phi_n(\zeta_5) $  \\
\hline
$1$ & $\zeta_5 - 1$  \\
\hline
$5$ & $0$ \\
\hline
$5p^k$ & $p$  \\
\hline
$q^k$ & $-(\zeta_5)^{\wedge}{(-1)^k}$  \\
\hline
$q^k n_1, q \nmid n_1$ & $(\zeta_5)^{\wedge}((-1)^k\varphi(n_1))$  \\
\hline
$n_2$ & $(-1)^{\varphi(n/5;n)} \zeta_5^{\varphi(n)/2} |1+\zeta_5|^{\wedge}((-1)^{\Omega(n)+1} 2^{\omega(n)-1})$ \\
\hline
$5 n_1$ & $e^{\Lambda(n_1)}/\Phi_{n_1}(\zeta_5)$\\
\hline
\end{tabular}
\end{center}
Here $p$ is $5$ or a prime number congruent to $1$ modulo $5$ and $q$ is a prime congruent to $-1$ modulo $5$. The integers $n_1, n_2 \ge 2$ are not divisible by $5$ and have the property that
\begin{itemize}
\item $p' \not\equiv 1 \pmod 5$ for every prime $p'$ dividing $n_1$;
\item $p' \not\equiv \pm1 \pmod 5$ for every prime $p'$ dividing $n_2$.
\end{itemize}
Furthermore, $k$ is an arbitrary positive integer.
\end{lem}
\begin{proof}
The first two entries of the table follow by 
direct computation and hence we may assume that $n \ne 1,5$.
\begin{itemize}
\item In case $\nu_5(n) \ge 2$ we have $\Phi_n(\zeta_5) = \Phi_{n/5}(1)$, which yields $5$ if $n = 5^{k+1}$ and 1 otherwise.
\item In case $\nu_5(n) = 0$ we separately consider three subcases:
\begin{enumerate}
\item The integer $n$ has a prime factor $p \equiv 1 \pmod 5$.\\
By Lemma \ref{mod1} we obtain $p$ if $n = 5 p^k$ and $1$ otherwise.
\item The integer $n$ has no prime factor $p \equiv 1 \pmod 5$ and it has a prime factor $q \equiv -1 \pmod 5$. There are two possibilities:
\begin{enumerate}[label=\roman*)]
\item $n = q^k$. Lemma \ref{mod-1}a yields $\Phi_n(\zeta_5) = - (\zeta_5)^{\wedge}{(-1)^{k}}$.
\item $n = q^k n_1'$ with $q \nmid n_1'$. Lemma \ref{mod-1}b implies $\Phi_n(\zeta_5) = (\zeta_5)^{\wedge}((-1)^k \varphi(n'))$.
\end{enumerate}
\item The integer $n$ has no prime factor $p \equiv \pm 1 \pmod 5$. \\
Corollary \ref{BB-C-values5} shows 
$|\Phi_n(\zeta_5)| = |1+\zeta_5|^{\wedge}((-1)^{\Omega(n)+1} 2^{\omega(n)-1})$. The value $\Phi_n(\zeta_5)$ is obtained by Lemma \ref{lem:phin-eval}.
\end{enumerate}
\item In case $\nu_5(n) = 1$ we have $\Phi_n(\zeta_5) = \Phi_{n/5}(1) / \Phi_{n/5}(\zeta_5) = e^{\Lambda(n/5)} / \Phi_{n/5}(\zeta_5)$. \qedhere
\end{itemize}
\end{proof}

\section{The logarithmic derivative $f_n(z)$ of $\Phi_n(z)$} \label{sec:log'(Phin)}

In this section we consider the logarithmic derivative 
$f_n(z)$ of $\Phi_n(z)$. Thus
$$f_n(z)= (\log\Phi_n(z))' = \frac{\Phi_n'(z)}{\Phi_n(z)}.$$

If we compute $f_n(\xi_m)$, then we can use the value $\Phi_n(\xi_m)$ to obtain $\Phi_n'(\xi_m)$. First, we calculate $f_n(\pm 1)$ with elementary methods. Later, we apply the ideas presented in Section \ref{sec:values:general-method} to develop a general method for computing $f_n(\zeta_m)$ when $(n,m) = 1$. Note that  as a consequence of \eqref{eq:simp-coprime} we can reduce the computation of $f_n(\zeta_m)$ to the case when $n$ and $m$ are coprime. Indeed for $n=n_1n_2$, where $n_1=\prod_{p^e \parallel n,~p\,\nmid \,m}p^e$, we have

\begin{equation} \label{eq:sim-coprime:deriv}
f_n(\xi_m) = \sum_{d\,\mid n_2}\mu(n_2/d) f_{n_1}(\xi_m^d).
\end{equation}

This method will be used to easily obtain $f_n(\zeta_m)$ for $m \in \{3,4,6\}$.

\begin{lem} \label{phin1:deriv}
We have $f_n(1)= \varphi(n)/2$ for $n > 1$ and
$f_n(-1)=-\varphi(n)/2$ for every $n\neq 2$.
\end{lem}
\begin{proof}
The proof follows from applying Lemma \ref{self-reciprocal:deriv} with $f=\Phi_n$ 
and $d=\varphi(n)$.
\end{proof}
\begin{cor} We have
\begin{center}
\begin{minipage}{0.4\textwidth}
\[\Phi_n'(1)=\begin{cases}1 & \hbox{if } n=1;\\
p\varphi(n)/2 & \hbox{if } n=p^e;\\
\varphi(n)/2 & \hbox{otherwise.}\end{cases}\]
\end{minipage}
\begin{minipage}{0.4\textwidth}
\[\Phi_n'(-1)=\begin{cases} 1 & \hbox{if } n=2;\\
-p\varphi(n)/2 & \hbox{if } n=p^e;\\
-\varphi(n)/2 & \hbox{otherwise.}\end{cases}\]
\end{minipage}
\end{center}
\end{cor}

\begin{Thm} \label{BB-T-derivative}
Let us assume that $n,m > 1$ are coprime. For all $\chi\in\widehat{G}(m)$ put
$$c_\chi(\xi_m) =
\sum_{g\in G(m)}\frac{\xi_m^{-1}\xi_m^g}{1-\xi_m^g}\overline{\chi}(g).$$
Then
$$f_n(\xi_m) = -\frac{n}{\varphi(m)}\sum_{\chi\in\widehat{G}(m)}
c_\chi(\xi_m) \chi(n)
\prod_{p\,\mid \,n}(1-\frac{\overline{\chi}(p)}{p}).$$
\end{Thm}

\begin{proof}
Logarithmic differentiation of \eqref{phimoebius} yields
$$f_n(z) = -\sum_{d\, \mid n}\mu\big (\frac{n}{d}\big )\frac{dz^{d-1}}{1-z^d}.$$
The function $\xi_m^{d-1}/(1-\xi_m^d)$ of variable $d$ can be treated as a function $G(m)\to\mathbb{C}$. Therefore for all $d\mid n$ we have
$$\frac{\xi_m^{d-1}}{1-\xi_m^d} = \frac{1}{\varphi(m)}\sum_{\chi\in\widehat{G}(m)}c_\chi(\xi_m)\chi(d).$$
Applying this to the above formula on $f_n$ we obtain
$$f_n(\zeta_m) =
-\frac{1}{\varphi(m)}\sum_{\chi\in\widehat{G}(m)}
c_\chi(\xi_m)\sum_{d\,\mid n}\mu(\frac{n}{d})d\chi(d)
= -\frac{1}{\varphi(m)}\sum_{\chi\in\widehat{G}(m)}
c_\chi(\xi_m)J_1(\chi;n).$$
The proof is completed by invoking 
Lemma \ref{lem:flauwmultiplicatief} with $k=1$.
\end{proof}

\begin{cor} \label{BB-C-derivative346}
Set $m \in \{3,4,6\}$ and $n>1$ coprime. We have
$$f_n(\xi_m)= \frac{\varphi(n)}{2\xi_m}\left(1 - (-1)^{\Omega(n_-)}\frac{1+\xi_m}{1-\xi_m}\prod_{p\,\mid\, n_-}\frac{p+1}{p-1}\right),$$
where $n_-$ is the product of the prime powers $p^k\parallel n$ 
with $p\equiv -1\pmod{m}$.
\end{cor}
\begin{proof}
\indent In case  $m\in\{3,4,6\}$ there are precisely two characters: the principal 
character $\chi_1$ and the non-principal character $\chi_2$. A
simple computation gives
$$c_{\chi_1}(\xi_m) = \frac{1-\xi_m^{-1}}{1-\xi_m}=-\xi_m^{-1}, \qquad c_{\chi_2}(\xi_m) = \frac{1+\xi_m^{-1}}{1-\xi_m}.$$
Theorem \ref{BB-T-derivative} yields
\begin{equation*}
f_n(\xi_m) = \frac{\xi_m^{-1}}{2}n\prod_{p\,\mid\, n}(1-\frac{1}{p}) - \frac{1+\xi_m^{-1}}{2(1-\xi_m)}n
\chi_2(n)\prod_{p\, \mid n}(1-\frac{\overline{\chi_2}(p)}{p}),
\end{equation*}
which is easily rewritten in the desired way by
noting that $\chi_2(n)=(-1)^{\Omega(n_-)}$.
\end{proof}

\section{The result of Vaughan}
\label{Vaughan}

We use Corollary \ref{BB-C-values5} to give an alternative proof of the following theorem by Vaughan \cite{Vaughan}.

\begin{Thm} \label{BB-T-Vaughan}
Let $A_n$ denote the height of $\Phi_n$. There exist infinitely many integers $n$ for which 
$$\log\log A_n \ge (\log2+o(1))\frac{\log n}{\log\log n}.$$
\end{Thm}

\begin{proof}
Let $x$ be large and $n$ be a product of all primes $p\le x$ satisfying $p\equiv\pm2\pmod5$. 
By two equivalent versions of the prime number theorem for
arithmetic progressions we have
$$\log n=\sum_{p\le x,~p\equiv \pm 2\pmod 5}\log p\sim \frac{x}{2},$$ respectively $$\omega(n)=\sum_{p\le x,~p\equiv \pm 2\pmod 5}1\sim \frac{x}{2\log x}.$$ 
It follows that
$\log \log n \sim \log x$ and so
\begin{equation}
\label{omegah}
\omega(n)\sim \frac{\log n}{\log\log n}
\end{equation}
as $x$ (and hence $n$) 
tends to infinity. 
Recall that by Corollary \ref{BB-C-values5} we have
$$\log|\Phi_n(\xi_5)| = (-2)^{\omega(n)-1}\log|1+\xi_5|.$$
One checks that there is a primitive fifth root of unity $\zeta$ for which
$\log|1+\zeta|>0$, but also one for which $\log|1+\zeta|<0$.
Thus we can choose a
primitive fifth root of unity $z_n$ for which $\log|\Phi_n(z_n)|>0$.
By Corollary \ref{BB-C-values5} and the asymptotic equality
\eqref{omegah} we infer that there is an $x_0$ such that
for all $x\ge x_0$ the corresponding $n$ satisfies
 $\log|\Phi_n(z_n)|>\log n$.
It follows that for $x\ge x_0$ (and hence $n$) tending to infinity
the asymptotic inequality
$$\log\log A_n \ge \log\log\left(\frac{|\Phi_n(z_n)|}{n}\right) = (\log2+o(1))\frac{\log n}{\log\log n}$$
holds true, where the first inequality is 
a consequence of \eqref{An}.
\end{proof}

\section{Application to Kronecker polynomials} \label{kronappl}

A \emph{Kronecker polynomial} is a monic polynomial with integer coefficients having all 
of its roots on or inside the unit disc. The following result of Kronecker relates Kronecker polynomials with cyclotomic polynomials.

\begin{lem}[Kronecker, 1857; cf. \cite{Kronecker}] \label{kroneckerpols}
If f is a Kronecker polynomial with $f(0) \ne 0$, then all roots of f are actually on the unit circle and f factorizes over the rationals as a product of cyclotomic polynomials.
\end{lem}
By this result  and the fact that cyclotomic polynomials are 
monic and irreducible we can factorize 
a Kronecker polynomial $f(x)$ into irreducibles as 
\begin{equation}
\label{cyclofactor}
f(x) = x^{e} \prod_{d \in {\mathcal D}}\Phi_d(x)^{e_d},
\end{equation}
with $e \ge 0$, ${\mathcal D}$ a finite set and each $e_d\ge 1$.
\begin{cor} \label{cor:reciprocity}
Let $f$ be a Kronecker polynomial with $f(0) \ne 0$. Let $k$ be such that $\Phi_1^k \parallel f$. If $k$ is even, then $f$ is self-reciprocal, otherwise $f$ is anti self-reciprocal.
\end{cor}
\begin{proof}
Recall that $\Phi_1$ is anti self-reciprocal and $\Phi_d$ is self-reciprocal for $d \ge 2$.
\end{proof}
In light of Corollary \ref{cor:reciprocity} one can apply the results of Section \ref{selfie} to Kronecker polynomials.

\begin{prop} \label{evalkronecker}
Let $f$ be a Kronecker polynomial with $f(0) \ne 0$. Then\\
{\rm a)} $f(1) \ge 0$.\\
{\rm b)} If $f(1) \ne 0$, then $f(-1) \ge 0$. Furthermore, if $f(-1) > 0$, then $f(x) > 0$ for all $x \in \mathbb{R}$.
\end{prop}
\begin{proof}$~$\\
{\rm a)} We have $f(1) \ge 0$ by \eqref{cyclofactor} and Lemma \ref{valueat1A}.\\
{\rm b)} If $f(1) \ne 0$, then $1 \not\in \mathcal{D}$. We 
have $\Phi_n(-1) \ge 0$ for every $n > 1$ by 
Lemma \ref{valueat-1}. Hence we obtain $f(-1) \ge 0$. Furthermore, if $f(-1) > 0$, then $2 \not\in \mathcal{D}$. Let $x \in \mathbb{R}$. We have $\Phi_n(x) > 0$ for every $n > 2$ and, consequently, $f(x) > 0$.\qedhere
\end{proof}

Using Lemma \ref{kroneckerpols} and the results of Section \ref{sec:low-order} one can obtain some information about the factorization and the values of Kronecker polynomials. 

\begin{lem} \label{lem:factors}
Let $m \in \{1,2,3,4,6\}$. 
Suppose that $f$ is of the form \eqref{cyclofactor}
and, moreover,
satisfies $\min \mathfrak{D}>m$. Then
$$|f(\xi_m)|=\prod_{d\in \mathfrak{D}\atop m|d,~\Lambda(d/m)\ne 0}
|\Phi_d(\xi_m)|^{e_d}
=\exp\big(\sum_{d\in \mathfrak{D},~m|d}e_d\Lambda(d/m)\big)\in 
\mathbb Z_{>0}.$$
\end{lem}
The following result is a reformulation of the latter, but with
$\mathfrak{D}$ assumed to be unknown.
\begin{lem}
\label{reformulation}
Let $f$ be a Kronecker polynomial and $m \in \{1,2,3,4,6\}$. 
Let us also assume that 
$f(\zeta_d) \ne 0$ for every $d\le m$. Then 
$|f(\xi_m)|$ is an integer and each of its 
prime factors $q$
is contributed by a divisor $\Phi_d$ of $f$ with
$d=mq^t$ for some $t\ge 1$.
\end{lem}
These lemmas are easily proved on 
using Lemma \ref{lem:tablecombi} and weaker 
versions of them have already been applied to cyclotomic numerical semigroups \cite{CGM}.\\

\noindent {\tt Acknowledgement}. We acknowledge helpful discussions with Peter Stevenhagen 
(mathematical content, presentation) and Lola Thompson (presentation, English)
of earlier versions of this paper. In particular, Peter Stevenhagen sketched 
the third author how to compute 
the resultant of two cyclotomic polynomials by 
purely algebraic number theoretical means. In this paper we presented another proof staying
closer to the basic definitions (Theorem \ref{thm:apostol}). We also 
thank Hendrik Lenstra for pointing
out some glitches in an earlier version.\\
\indent A substantial part of this paper was written during a one month
internship in the autumn of 2016 of the second author at the Max Planck Institute for Mathematics in
Bonn. He would like to thank the third author for the opportunity given and the staff for their hospitality. 
He would also like to acknowledge the third author and Pedro A. Garc\'ia-S\'anchez for their teachings and guidance.

\medskip\noindent Bart{\l}omiej Bzd\c{e}ga  \par\noindent
{\footnotesize Faculty of Mathematics and Computer Science, Adam Mickiewicz 
University, Umultowska 87, 61-614 Poznan, Poland.\hfil\break
e-mail: {\tt exul@amu.edu.pl}}

\medskip\noindent Andr\'es Herrera-Poyatos \par\noindent
{\footnotesize Faculty of Science, University of Granada,  Avenida de la Fuente Nueva, 18071 Granada, Spain.\hfil\break
e-mail: {\tt andreshp9@gmail.com}}

\medskip\noindent Pieter Moree \par\noindent
{\footnotesize Max-Planck-Institut f\"ur Mathematik,
Vivatsgasse 7, D-53111 Bonn, Germany.\hfil\break
e-mail: {\tt moree@mpim-bonn.mpg.de}}

\end{document}